\documentclass[11pt,a4paper]{amsart}
\usepackage[latin1]{inputenc}
\usepackage{amsxtra}
\usepackage{amssymb}
\usepackage{amsthm}
\usepackage{euler}
\usepackage{graphicx}
\input{xy}
\xyoption{all}
\usepackage[arrow,curve,matrix]{xy}

\evensidemargin=\oddsidemargin
\makeatletter
\def\cleardoublepage{\clearpage\if@twoside \ifodd\c@page\else
  \hbox{}
  \thispagestyle{empty}
  \newpage
  \if@twocolumn\hbox{}\newpage\fi\fi\fi}
\makeatother
\usepackage{fancyhdr}
\headwidth=\textwidth

\theoremstyle{plain} 

\newtheorem{theorem}{Theorem}[section]
\newtheorem*{teo}{Theorem}

\newtheorem{corollary}[theorem]{Corollary}

\newtheorem{proposition}[theorem]{Proposition}
\newtheorem*{prop}{Proposition}

\newtheorem{lemma}[theorem]{Lemma}
\newtheorem*{lema}{Lemma}

\theoremstyle{remark}
\newtheorem*{remark}{Remark}

\theoremstyle{definition}
\newtheorem*{definition}{Definition}

\newtheorem*{notation}{Notation}
\newtheorem{example}{Example}

\newtheorem*{assumption}{Assumptions}

\newcommand{\diamantej}{\underset{j}{\diamond}}
\newcommand{\diamantei}{\underset{i}{\diamond}}
\newcommand{\tensorQ}{\underset{kQ_0}{\otimes} }
\newcommand{\gl}{{\mathfrak{g} \mathfrak{l}}}
\newcommand{\g}{{\mathfrak{g}}}

\newcommand{\matrizDcero}{\scalebox{0.7}{$\begin{pmatrix} 0&0\\D_0&0 \end{pmatrix}$}}
\newcommand{\matrizDuno}{\scalebox{0.7}
{$\begin{pmatrix} 0&0\\ D_1&0\end{pmatrix}$}}

\title{On the semisimplicity of the outer derivations of monomial algebras.}
\author{Selene S{\'a}nchez-Flores}
\address{Mathematisches Institut, Universität zu Köln, Weyertal 86-90, 50931 Köln.}
\email{ssanchez@math.uni-koeln.de}
\date{}
\begin{document}
\maketitle
\begin{abstract}
We show that the Hochschild cohomology 
of a monomial algebra over a field of characteristic zero 
vanishes from degree two    
if the first Hochschild cohomology 
is semisimple as a Lie algebra.
We also prove that first Hochschild cohomology of 
a radical square zero algebra 
is reductive as a Lie algebra. 
In the case of the multiple loops quiver,  
we obtain the Lie algebra of square matrices of size equal to the number of loops. 
\end{abstract}
$\quad$ 
\section{Introduction.}
Let $A$ be an associative unital $k$-algebra where $k$ is a field.    
We denote by $HH^n(A)$ the Hochschild cohomology vector space in degree $n$ of 
$A$ with coefficients in the bimodule $A$, i.e.
$$
HH^n(A)=HH^n(A,A)=Ext^n_{A^e}(A,A)
$$
where $A^e=A \otimes A^{op}$ is the enveloping algebra of $A$. 
It is well known that 
$HH^1(A)$ is the vector space of outer derivations of $A$. 
Hence $HH^1(A)$ has a Lie algebra structure given by the commutator bracket. 
In \cite{gersten}, Gerstenhaber defined a bracket 
which generalizes the commutator. 
The Gerstenhaber bracket induces a graded Lie algebra structure on $HH^{*+1}(A)$. 
Therefore, $HH^n(A)$ 
is a Lie module over the Lie algebra $HH^1(A)$. 

In the present paper, we consider 
the Lie algebra structure on 
$HH^1(A)$ where $A$ is a finite dimensional monomial algebra over a field of characteristic zero. 
A {\it{monomial algebra}} is the quotient of a path algebra of a quiver $Q$ 
by an ideal generated by a set of paths of length at least two. 
In \cite{happel} and \cite{Bardzell}, 
a minimal projective resolution of $A$ over its enveloping algebra 
is described in terms of the combinatorics of the quiver. 
This resolution allows to compute the Hochschild cohomology. 
Moreover, it enables to study the Lie algebra of outer derivations space with combinatorial tools. 

The Lie algebra structure of $HH^1(A)$ where $A$ is a monomial algebra has been already studied in \cite{strametz}.   
In particular, Strametz provides sufficient and necessary conditions to the combinatorial data of $A$ in order to guarantee the semisimplicity on $HH^1(A)$. In this case, the semisimple Lie algebra obtained is isomorphic to a direct product of Lie algebras of type $A_n$ (i.e. trace zero square matrices). On the other hand, finite dimensional mo${}$dules over these Lie algebras are well known from representation theory. So, a natural question arise: {\it{What is the description of $HH^n(A)$ as a Lie module over the Lie algebra $HH^1(A)$, when this one is semisimple?}}

In this article, we answer the above question for monomial algebras over a field of characteristic zero: {\it{the Hochschild cohomology from degree two vanishes if the first Hochschild cohomology is semisimple as a Lie algebra}}. 
This answer brings out some other questions concerning the cohomology structure. 
For in${}$stance, the pursuit of examples where the Lie module structure of the Hochschild cohomology groups is not trivial.  
This has led us to consider the case where the Lie algebra in the first degree is not semisimple. 

Our first step was to study the case of radical square zero monomial algebras,  
since its Hochschild cohomology has been described explicitly in \cite{cibils}. 
This description allowed us to give an example where the Lie module structure is not trivial.  
In \cite{sele}, we described the Lie module structure of Hochschild cohomology vector spaces 
of the radical square zero algebra $A$ associated to the two loops quiver. 
In fact, we showed that $HH^1(A)$ is isomorphic to $gl_2\, k$, the Lie algebra of square matrices of size two 
and then we describe $HH^n(A)$ as a Lie module over $gl_2 \,k$ 
by giving the decomposition into a direct sum of irreducibles modules over $sl_2 \, k$, 
the trace zero matrices.  

In the present paper, we prove a more general result concerning the first Hochschild cohomology: 
{\it{The Lie algebra of the outer derivations of a radical square zero monomial algebra is reductive}}
(i.e. it is the direct product of a semisimple Lie algebra 
and an abelian Lie algebra). 
In the particular case of the multiple loops quiver, 
we obtain that $HH^1(A)$ is the Lie algebra of square matrices of size equals to the number of loops.

This paper is organized in three sections. 
In the first section, we begin by recalling the combinatorial description 
of both: the first Hochschild cohomo${}$logy vector space  
and the commutator bracket. 
The description of the Hochschild cohomology is obtained from  
the complex induced after applying the 
functor ${Hom_{A^e}(-,A)}$ to the Happel-Bardzell projective resolution (see \cite{Bardzell}). 
Once this complex is simplified,  
the space of cochains is expressed in terms of parallel paths and 
the differential maps are expressed in terms of an operation 
that replaces para${}$llel arrows in a path. 
In order to describe the commutator bracket in terms of parallel paths, 
Strametz gave maps from the first Hochschild cochain to the space of parallel paths 
(see \cite{strametz}). 
Those maps induce inverse linear isomorphism at the cohomological level, 
which enables the translation of the commutator bracket in terms of parallel arrows. 

In section 2, using the combinatorial commutator, 
we specify the  Lie algebra structure of $HH^1(A)$ in two cases: 
when $A$ is a radical square zero monomial algebra in one hand  
and 
when $A$ is a triangular complete monomial  algebra 
in the other hand. 
Recall that a finite dimensional algebra is called {\it{triangular}} 
if the quiver has no oriented cycles. 
A {\it{complete monomial algebra}} 
is a monomial algebra that verifies the following property:  
every path of length at least two parallel 
to a path which is zero in the algebra is also zero. 
For instance, radical square zero monomial algebras 
are complete monomial. 
We will study both cases separately: 
radical square zero algebras without any restriction on the quiver 
in one hand and 
complete monomial algebras 
whose quiver contains no oriented cycles in the other hand. 

Keeping in mind Levi's decomposition theorem, 
we will compute the sol${}$vable radical of $HH^1(A)$  
in order to obtain the semisimple part. 
Recall that the {\it{semisimple part of a Lie algebra}} is 
precisely the quotient by its solvable radical.  
For radical square zero monomial algebras, 
we specify the semisimple part and the solvable radical  
using the combinatorics of the quiver. Let us introduce some notation. 
Given a quiver $Q$, 
we denote $\overline{Q}$ the quiver obtained 
by identifying parallel arrows, 
i.e. multiple parallel arrows in $Q$ 
are seen as only one arrow in $\overline{Q}$. 
Denote $S$ the set of arrows in $\overline{Q}$ 
that correspond to more than one arrow in $Q$. 
We denote by $|\alpha|$ the cardinality of an element $\alpha$ in $S$.

One of the main results of this paper is the following. 
Let A be a radical square zero monomial algebra 
over a field of characteristic zero. 
The solvable radical of $HH^1(A)$ is abelian 
and its dimension is equals to the Euler characteristic of the quiver $\overline{Q}$. 
Moreover, the semisimple part of $HH^1(A)$    
is a direct product of Lie algebras of type 
$A_{|\alpha|}$ where $\alpha$ is in $S$.  
For example, 
consider the monomial algebra of radical square zero associated to 
the multiple-loops quiver, this is $k[x_1, \dots x_r]/<x_ix_j>_{i,j=1}^r$ 
where $r$ is the number of loops. 
As a consequence of the result mentioned above,  
its first Hochschild cohomology 
is $gl_r (\, k \,)$, the Lie algebra 
of square matrices of size $r$. 

Next, we consider 
triangular complete monomial algebras. 
The semisimple part of 
the first Hochschild cohomology Lie algebra 
can be expressed in the same terms as in the previous case. 
The solvable radical is not necessary abelian, however 
we give a description of it in terms of the combinatorics of the quiver. 

In the third section, 
our main purpose is to prove that for monomial algebras over a field of characteristic zero,  
the Hochschild cohomology vanishes from the second degree 
if the first Hochschild cohomology is semisimple as a Lie algebra.  
To do so, we begin by recalling the necessary and sufficient conditions   
for semi-simplicity of the the Lie algebra $HH^1(A)$ where $A$ 
is monomial, which are given in \cite{strametz}. Then we restate these conditions as follows. 

Let us assume that ${A=kQ/<Z>}$ where 
$kQ$ is the path algebra of a quiver $Q$ 
and $<Z>$ is the two-sided ideal generated by a minimal set of paths of length two which we denote by $Z$. The assumption of "minimal" is not restrictive since given a set of paths, we can extract a subset which is minimal and that generates the same ideal. The restatement of Strametz' theorem is the following. 
The Lie algebra $HH^1(A)$ is semisimple if and only if {\it{$Z$ is closed under parallel paths}}, the underlying graph of $\overline{Q}$ is a tree and the set $S$ is not empty.  
We also prove that if {\it{$HH^1(A)$ is semisimple then $A$ is a complete monomial algebra}}. 

Furthermore,  
we compute the Hochschild cohomology vector spaces of a monomial algebra 
$A=kQ/<Z>$ where $\overline{Q}$ is a tree and $Z$ is closed under parallel paths. 
To do so, we use the Happel-Bardzell projective resolution and the results of section two about complete monomial algebras.  
We prove that the Hochschild cohomology vector spaces of degree at least two vanish  
under these hypotheses.  
Finally, as a consequence of this computation, we obtain that if $HH^1(A)$ is semisimple then ${HH^n(A)=0}$ for all $n \geq 2$. \\

\thanks{{\bf{Acknowledgment.}} This paper is part of my PhD thesis that I defended in 
the University of Montpellier 2. I am thankful to my advisor Claude Cibils 
for discussions and advices. I would also like to thank Andrea Solotar for 
helpful suggestions in improving this paper. }
\setcounter{section}{0}
\section{The space of outer derivations of a monomial algebra.}
It is well known  
that the first Hochschild cohomology vector space  
is the space of outer derivations.  
Clearly, it is endowed with a Lie algebra structure 
given by the comm${}$utator bracket. 
Such structure was studied by C. Strametz 
for the case of finite dimensional monomial algebras, using combinatorial tools. 
In this section, we will recall the combinatorial description of 
the space of outer derivations of a monomial algebra.  
Then we will remind the description of the commutator 
bracket given by Strametz in the same terms as the combinatorial realization 
of the first Hochschild cohomology group of a monomial algebra. 

\subsection{Notations and assumptions.}
In this section we fix the notation and assumptions that we will be using all along this paper. 

Given a {\it{quiver}} $Q$, we denote 
$Q_0$ the set of {\it{vertices}} and 
$Q_1$ the set of {\it{arrows}}.  
The applications 
{\it{source}} and {\it{target}} 
defined from the set of arrows to the set of vertices
are denoted by $s$ and $t$ respectively, 
The {\it{loop}} quiver is given by one vertex and one arrow. 
The {\it{multiple loops}} quiver is given by one vertex and several loops. 
The  {\it{multiple arrows}} quiver 
is given by the following data:
$Q_0=\{e_1,e_2 \}$, 
$Q_1=\{a_1,\dots  a_r \}$, 
$s(a_i)=e_1$ and $t(a_i)=e_2$ for $i=1,\dots,r $. 

The {\it{paths}} are constructed 
by concatenating arrows as follows. 
Let $a_1, \dots, a_n$ be arrows such that 
$s(a_i)=t(a_{i+1})$ for $i=1,\dots,n-1$, 
the expression $a_1 a_2 \cdots a_n$ is a path denoted $p$. 
Let us remark that the source map and the target map 
can also be defined for paths as follows: $s(p)=s(a_n)$ and $t(p)=t(a_1)$. 
Let $p$ and $q$ be two paths, 
we say that they are {\it{composable}} if and only if $s(p)=t(q)$, 
in this case we write $pq$ for the path obtained after concatenation. 
Besides, a path $c$ such that $s(c)=t(c)$ is called an {\it{oriented cycle}}. 
Moreover,  the {\it{length}} of a path is the number of arrows used in its expression in arrows. 
We will denote $Q_n$ the set of all paths of $Q$ of length $n$. 
The set of vertices $Q_0$, 
which is the set of paths of length zero, 
will be considered as the set of {\it{trivial paths}}.  
Let $p$ and $q$ be two paths. 
We say that {\it{$q$ divides $p$}} if there exist paths $x$ and $y$ such that 
$p=xqy$. 
The {\it{underlying graph}} of the quiver is the graph 
obtained when the orientations of the arrows are ignored.

Let $k$ be a field. 
We denote $kQ$ the {\it{path algebra}} of a quiver $Q$.
Recall that its underlying vector space has a basis given by all paths of the quiver,  
and the multiplication is defined by concatenation of paths 
whenever they are composable, and zero otherwise. 
An {\it{admissible ideal}} $I$ is a two-sided ideal of a path algebra of a quiver $Q$ 
such that $$<Q_n> \, \subseteq \, I \subseteq \, <Q_2> $$
for some $n$, 
where $<Q_i>$ is the two-sided ideal generated by $Q_i$. 
For example, the two-sided ideal generated by $Q_n$ with $n \geq 2$ is an admissible ideal. Another example is the zero ideal which is admissible 
if and only if the quiver has no oriented cycles.  

A finite dimensional {\it{monomial algebra}} 
is a quotient of the path algebra of a quiver $Q$ 
by an admissible ideal generated 
by a set of paths of length at least two, 
which we will denote $Z$. 

\begin{assumption}
In the sequel, we will assume that 
\begin{itemize}
\item[-] $Q_1$ and $Q_0$ are non-empty sets. 
\item[-] $Q$ is finite, i.e. $Q_0$ and $Q_1$ are finite sets. 
\item[-] $Q$ is a connected, i.e. 
the underlying graph of Q is connected.  
\item[-] $Z$ is minimal, this means that for all path $p$ in $Z$ 
and for all path $q \ne p$ 
that strictly divides the path $p$, $q$ does not belong to the set $Z$. 
\end{itemize}
This last assumption is not restrictive since  
we can always extract from a set of paths 
a minimal subset  such that both sets generate the same ideal.  

Let
$B$ be the set of paths of $Q$ 
which are not divided by any element of $Z$. 
It is clear that the elements of $B$ 
form a basis of the monomial algebra $A$. 
Moreover, the Jacobson radical 
(i.e. the intersection of all left maximal ideals) 
of a monomial algebra denoted 
$r=rad \, A$ is given by 
$$r  = \frac{<Q_1>}{<Z>} $$
where $<Q_1>$ is the two-sided ideal generated by $Q_1$.  
Furthermore, $E=kQ_0$ is isomorphic to $A/r$. 
Moreover $A\cong E \oplus r$, as predicted by the Wedderburn-Malcev theorem.  
\end{assumption}

\subsection{A description of the Hochschild cohomology in first degree.} 
The computation of Hochschild cohomology vector spaces 
for a monomial algebra has been done 
using the Happel-Bardzell projective resolution \cite{Bardzell}. 
In the complex obtained,  
the space of cochains is expressed in terms of parallel paths and 
the differential maps are expressed in terms of an operation 
that replaces parallel arrows in a path. 
Let us first introduce both tools: parallel paths and this operation. 
 
Given a quiver $Q$, we say that two paths $\alpha$ and $\beta$ 
are {\it{parallel}} if and only if 
they have the same source and the same target.  
If $\alpha$ and $\beta$ are parallel  
we write $\alpha \parallel \beta$. 
Let $X$ and $Y$ be sets consisting of paths of $Q$, 
then  $X \parallel Y$ denotes the set of all pairs of paths 
$(\alpha, \beta)$ in $X \times Y$ that are parallel. 
We denote 
$k(X \parallel Y)$ the $k$-vector space 
generated by $X \parallel Y$.

Now, we introduce the operation $\diamantei$ that 
replaces parallel arrows in a path,  
where $i$ is a natural number between 1 
and the length of the path. 
Given a path $\alpha$ in $Q_n$ with $n \geq 1$
such that 
${\alpha =  a_1 \cdots a_i \cdots a_n}$,  
fix a natural number $1\leq i \leq n$. 
Let $\beta$ be a non trivial path in $Q_m$ such that 
$a_i \parallel \beta$. 
Replacing the arrow $a_i$ by the path $\beta$ we obtain a path 
\[
a_1 \cdots a_{i-1} \beta a_{i+1} \cdots a_n. 
\]
in $Q_{n+m-1}$. 
Let us denote $k(Q_n)$ the vector space generated by all paths of length $n$, 
the operation $\diamantei$ is given by: 
\[
\begin{array}{cccl}
 \diamantei:&k(Q_n) \times k(Q_m) & \rightarrow & k(Q_{n+m-1}) \\
 \,&\,&\,&\\
 \, &(\alpha,\beta)&\mapsto&
 \alpha \diamantei \beta = 
 \begin{cases} 
a_1 \cdots a_{i-1} \beta a_{i+1} \cdots a_n  & \text{if $a_i \parallel \beta$} \\ 
 0 & \text{otherwise.}
 \end{cases}
\end{array}
\]
\begin{notation}
Let $\alpha$ be a path in $Q$ and $(a,\gamma)$ in $Q_1 \parallel B$. 
Following Strametz we denote $\alpha^{(a,\gamma)}$ the element in $A$ 
given by the sum of all nonzero paths 
obtained by replacing each appearance of the arrow 
$a$ in $\alpha$ by $\gamma$. 
If the path $\alpha$ does not contain the arrow $a$ or if every replacement 
of $a$ in $\alpha$ is not a path in $B$, 
then $\alpha^{(a,\gamma)}=0$. 
For example, let $\alpha=aba$ be a path,
$\alpha^{(a,\gamma)}=ab \gamma + \gamma ba$ 
in case $ab \gamma$ and 
$\gamma ba$ are paths in $B$. 
In general, if $\alpha =a_1 \cdots a_i \cdots a_n$, 
then  $\alpha^{(a, \gamma)}$ is the element of $A$ given by 
\[
\alpha^{(a,\gamma)} =
\displaystyle{\sum_{i=1}^n 
\delta_{a_i}^a\:
\chi_{B}(\alpha \diamantei \gamma) \: 
\alpha \diamantei \gamma } 
\]
where $\delta_{a_i}^a$ is the Kronecker symbol  
and $\chi_{B}$ is the characteristic function. 
It is clear that $\alpha$ is parallel to $\alpha \diamantei \gamma$ for all $i$. 
Now, let us suppose $\alpha$  belongs to a certain set of paths $X$. 
We denote  $(\alpha,\alpha^{(a,\gamma)})$ the element 
in $k(X \parallel B)$ given by the following sum:
\[
(\alpha,\alpha^{(a,\gamma)}): =
\displaystyle{\sum_{i=1}^n 
\delta_{a_i}^a\:
\chi_{B}(\alpha \diamantei \gamma) \: 
(\alpha, \alpha \diamantei \gamma )} .
\]
\end{notation}

We are able to state 
the proposition providing a combinatorial description of the vector space of outer derivations $HH^1(A)$ when $A$ is monomial. 

\begin{prop}[\cite{strametz}]
Let $A=kQ/<Z>$ be a monomial algebra and let 
$B$ denote the basis of $A$ described before. 
The beginning of the complex induced by the Happel-Bardzell resolution 
can be described in the following way:
\[
0 \longrightarrow k(Q_0 \parallel B) 
\stackrel{\psi_0}{\longrightarrow} k(Q_1 \parallel B)
\stackrel{\psi_1}{\longrightarrow} k(Z   \parallel B) 
\rightarrow \cdots
\]
where the maps $\psi_0$ and $\psi_1$ are given by
\begin{equation}
\begin{array}{ccl}
\psi_0(e,\gamma)&=
&\displaystyle{
\sum_{a \, | \, s(a)=e} (a, a\gamma) \: \: - \sum_{a \, | \, t(a)=e} (a, \gamma a )} 
\end{array}
\end{equation}
\begin{equation}\label{psiuno}
\begin{array}{ccll}
\psi_1(a,\gamma)&=& 
\displaystyle{\sum_{p \in Z} \; (p,p^{(a,\gamma)}) }. & \quad
\end{array}
\end{equation}
Therefore, 
$$
HH^1(A) \cong \frac{Ker \, \psi_1 }{ Im \, \psi_0}  \; .
$$
\end{prop}


\subsection{Combinatorial commutator bracket.}
The Lie algebra structure of the outer derivations is given by the commutator bracket. In order to study this structure using the above 
combinatorial presentation we translate the commutator bracket 
in terms of the combinatorial description of $HH^1(A)$. 
In order to translate the commutator bracket Strametz gave  
maps from $C^1(A,A)$ to ${k(Q_1 \parallel B)}$ and vice versa. Those maps induce inverse linear isomorphisms at the cohomological level, i.e. between $HH^1(A)$.  In fact, they are induced from some comparison maps, which are written explicitly for the first degree. 
This procedure enables us to give a Lie algebra structure 
to the vector space $k(Q_1 \parallel B)$. 
The combinatorial commutator bracket is given by the following result.

\begin{teo}[\cite{strametz}]
Let $A=kQ/<Z>$ be a monomial algebra and 
let $B$ be the corresponding basis of $A$. 
Consider the bracket 
$$
[\, - \, , \, - \,]_S: k(Q_1 \parallel B) \times k(Q_1 \parallel B) 
\longrightarrow  k(Q_1 \parallel B)
$$
given by 
\begin{equation}
\begin{array}{ccl}
[ \, (a,\alpha) \, , \, (b, \beta) \,]_S &=& 
(b,\beta^{(a, \alpha)}) - (a,\alpha^{(b, \beta)}) 
\\
\, &= &\displaystyle{ 
\sum_{i=1}^m 
\delta_{b_i}^a \, 
\chi_B(\beta \diamantei  \alpha)\, 
(b,\beta \diamantei  \alpha)} \\
\, &\, &\, \\
\,& - &
\displaystyle{
\sum_{i=1}^n \delta_{a_i}^b \, 
\chi_B(\alpha \diamantei  \beta) 
\, (a,\alpha \diamantei \beta)}
\end{array}
\end{equation}
where $\alpha=a_1 \cdots a_n$ and 
$\beta=b_i \cdots b_m$; 
the functions  $\delta_{x}^y$ and $\chi_B$ 
are the Kronecker symbol 
and the characteristic function respectively.

The above bracket induces a Lie algebra structure 
on  the first Hochschild cohomology group 
$
HH^1(A) \cong Ker \, \psi_1 / Im \, \psi_0
$
which is a Lie algebra isomorphic to 
$HH^1(A)$ with the commutator bracket. 
\end{teo}

\section{Lie algebra structure.}
In this section we are concerned 
about the Lie algebra structure of the first 
Hochschild cohomology of a monomial algebra. 
We will determine its radical and 
its semisimple part in two cases. 
The first one is when the  
monomial algebra is of radical square zero. 
In the second case, 
we will consider triangular complete monomial algebras. 
Then, using Levi's decomposition theorem, 
we obtain a complete description of 
the Lie algebra structure of the 
first Hochschild cohomology of such algebras.  

\subsection{Parallel arrows.}
The combinatorial commutator bracket 
described in the above section 
induces a Lie algebra structure 
on the cochains $k(Q_1 \parallel B)$. 
Let us remark that $k(Q_1 \parallel Q_1)$ 
becomes a Lie subalgebra 
with the combinatorial commutator bracket 
\[
[\,- \, , \, -\,]_S: k(Q_1 \parallel Q_1) \times k(Q_1 \parallel Q_1)
\rightarrow k(Q_1 \parallel Q_1)
\]
given by 
\[
[\, (a,a')\, , \, (b,b')\,]_S= 
\delta_{b'}^ a \, (b,a') - \delta_{a'}^b \, (a,b') \, .
\]
We call $k(Q_1 \parallel Q_1)$ the  {\it{Lie algebra of parallel arrows}}. 
In this subsection we will study both Lie algebras: 
$k(Q_1 \parallel Q_1)$ and $k(Q_1 \parallel B)$. 

Given a quiver $Q$ we have that $\parallel$ 
is an equivalence relation on the set of arrows $Q_1$. 
We denote $\overline{Q_1}$ the set of equivalence classes. 
It is clear that the maps source $s$ and target $t$ are well defined on $\overline{Q}_1$. 
The quiver which has $Q_0$ as vertices  
and $\overline{Q_1}$ as set of arrows, 
together with the maps $s$ and $t$ will be denoted $\overline{Q}$. 
Note that in the quiver $\overline{Q}$, 
all multiple parallel arrows of $Q$ are identified. 

We will show that the Lie algebra $k(Q_1 \parallel Q_1)$ can be expressed 
as a direct product of the endomorphism Lie algebras 
$\gl_{\alpha}$ where $\alpha$ is an arrow of $\overline{Q}$. 
Let us introduce such Lie algebras.
\begin{notation}
Given $\alpha$ in $\overline{Q}_1$ we denote 
$$
\gl_{\alpha}= 
\displaystyle{ \bigoplus_{a,a' \in \alpha} k \, (a,a')}
$$ 
and  
$$I_{\alpha}=\sum_{a \in \alpha} (a,a).$$
\end{notation} 
Clearly, the vector space $\gl_{\alpha}$ together 
with the above bracket is a Lie subalgebra 
of $k(Q_1 \parallel Q_1)$. 
Let us show that these are endomorphism Lie algebras. 
Denote $V_\alpha$ the vector space whose basis is the set $\alpha$, 
so $dim_k (V_\alpha)=|\alpha|$. 
Consider $End_{k}\,(V_\alpha)$, 
the Lie algebra of endomorphisms of 
$V_\alpha$ with the commutator bracket. 
We will show that $End_{k}\,(V_\alpha)$ and 
$gl_\alpha$ are isomorphic. 
Given $a,a' \in \alpha$, 
denote $f_{(a,a')}:V_\alpha \rightarrow V_\alpha$ the linear morphism
given by: 
\[
f_{(a,a')}(\displaystyle{\sum_{x \in \alpha}} \lambda_x x)= \lambda_{a}a'.
\]
The inverse map is given by the following: 
\[
\begin{array}{clc}
\gl_{\alpha} & \rightarrow & End_k \, (V_\alpha)\\
(a,a')& \mapsto& -f_{(a,a')}
\end{array}
\]
The minus sign in the right side is needed 
in order to guarantee this map to be an homomorphism of Lie algebras.  
Therefore, $\gl_\alpha$ 
is isomorphic to the Lie algebra of endomorphisms of $V_{\alpha}$. 
We will use this Lie algebra to describe the Lie algebra structure of the parallel arrows.  
The description of the Lie algebra $k(Q_1 \parallel Q_1)$ 
and its radical is given by the following lemma. 
\begin{lemma}\label{matrices} 
$$\displaystyle{k(Q_1 \parallel Q_1) = \prod_{\alpha \in \overline{Q}_1}
\gl_{\alpha} } $$
as Lie algebras. Moreover, 
\[
\displaystyle{ rad \, k(Q_1 \parallel Q_1) 
= \prod_{\alpha \in \overline{Q}_1} k \, I_{\alpha} } .
\]
\end{lemma}
\begin{proof}
If $\alpha \ne \beta$ then 
$\gl_{\alpha} \cap \gl_{\beta} =0$ and 
$
[(a,a'), (b,b')]_S=0 
$
for all $(a,a')$ in $\gl_{\alpha}$ and $(b,b')$ in $\gl_{\beta}$. 
Then it is easy  to conclude that $k(Q_1 \parallel Q_1)$ is the product of all $\gl_\alpha$. 
Now, recall that taking radical of Lie algebras commutes with finite products, 
so the radical of $k(Q_1 \parallel Q_1)$ is the product of the radicals 
of the $\gl_{\alpha}$'s. 
Since the radical of $\gl_{\alpha}$ is $k \, I_{\alpha}$, 
we obtain that the radical of $k(Q_1 \parallel Q_1)$ is 
the direct product of $k\, I_{\alpha}$'s. 
\end{proof}

Now we are ready to study 
the Lie algebra $k(Q_1 \parallel B)$. 
Let us remark that  
$$k(Q_1 \parallel B) = k(Q_1 \parallel Q_0) 
\oplus k(Q_1 \parallel Q_1) \oplus 
\bigoplus_{i=2}^N k(Q_1 \parallel B \cap Q_i) $$
where $N$ is the maximum length of non-zero paths in $A$. 
If we set that the elements of 
${k(Q_1 \parallel B \cap Q_i)}$ are of degree $i-1$, 
the combinatorial commutator bracket is graded. 
Let 
\[
R=\bigoplus_{i=2}^N k(Q_1 \parallel B \cap Q_i).
\] 

\begin{lemma}\label{decoqb}
Let $Q$ be a quiver without loops and 
consider the Lie algebra ${k(Q_1 \parallel B)}$. 
Then $R$ is a solvable ideal. 
Since $k(Q_1 \parallel Q_1)$ is a subalgebra 
of ${k(Q_1 \parallel B)}$, 
the following decomposition of 
$k(Q_1 \parallel B)$ holds
\[
k(Q_1 \parallel B)=k(Q_1 \parallel Q_1) \oplus R.
\]
\end{lemma}
\begin{proof}
First we prove that $R$ is an ideal: 
let $(x,\alpha)$ be in $k(Q_1 \parallel B)$ 
and let $(y, \beta)$ be  in $R$ (i.e. $l(\beta) \geq 2$). 
Using the definition of the combinatorial commutator bracket, 
$[\, (x, \alpha) \, , \, (y,\beta)\, ]_S$ is in 
$k(Q_1\parallel B \cap Q_{l(\alpha)+l(\beta)-1})$ where 
${l(\alpha)+l(\beta)-1 \geq 2}$. 
So it is clear that $R$ is an ideal. 
Let us prove that $R$ is solvable, i.e. that  the terms of 
its derived series ${\mathscr{D}}^l(R)$ vanish for $l \geq l_0$, 
where $l_0$ is some natural number. 
Recall that if ${\mathfrak{g}}$ is a Lie algebra 
then its derived series is the sequence defined by 
${\mathscr{D}}^0({\mathfrak{g}})={\mathfrak{g}}$ and 
${\mathscr{D}}^{l+1}({\mathfrak{g}})=
[{\mathscr{D}}^l({\mathfrak{g}}) \:, \:
{\mathscr{D}}^l({\mathfrak{g}})]$.
In order to prove that $R$ is solvable, let us remark the following: 
$${\mathscr{D}}^{1}(R)=[\, R \, , \, R \,]_S \subseteq 
\bigoplus_{i=3}^N k(Q_1 \parallel B \cap Q_i).$$ 
Moreover, the derived series of $R$ satisfies 
$${\mathscr{D}}^{l+1}(R)=
[{\mathscr{D}}^l(R) \:,\: 
{\mathscr{D}}^l(R)]_S \subseteq 
\bigoplus_{i=i_{l+1}}^N k(Q_1 \parallel B \cap Q_i)$$
where $i_{l} \leq i_{l+1} \leq N$. 
Then it is clear that $R$ is solvable since there exists $i_0$ 
such that $B \cap Q_i$ is empty for $i>i_0$. Therefore 
${\mathscr{D}}^{l}(R)=0$ for $l \geq l_0$.
\end{proof}

The following lemma allows to compute the radical of 
$k(Q_1 \parallel B)$. 
\begin{lemma}\label{lemaqarbol}
Let $Q$ be a quiver without loops. 
Consider the Lie algebra ${k(Q_1 \parallel B)}$. 
Then 
$rad \, k(Q_1 \parallel Q_1) \oplus R $
is a solvable ideal. 
Therefore, 
$$rad \, k(Q_1 \parallel Q_1) \oplus R \subseteq rad \, k(Q_1 \parallel B).$$
\end{lemma}

\begin{proof}
Let
$I=rad \, k(Q_1 \parallel Q_1) \oplus R =
\prod_{\alpha \in \overline{Q}_1} k \, I_{\alpha} \oplus R$. 
First, we will show that $I$ is an ideal. 
Since we have already proved that $R$ is an ideal, 
it is enough to verify that 
$[\, k(Q_1 \parallel B) \, , \, 
\prod_{\alpha \in \overline{Q}_1} k \, I_{\alpha}\, ]_S$ belongs to $I$. 
Let $(x, \gamma)$ be in $k(Q_1 \parallel B)$, we will calculate 
$[\, (x, \gamma) \, , \, I_\alpha \, ]_S$ for all 
$\alpha$ in $\overline{Q}_1$. 
If $l(\gamma)=1$, 
this means that $\gamma$ is an arrow which is parallel to $x$. 
So, if $x \notin \alpha$ then it is clear that 
$[\, (x,\gamma)\, , \, I_{\alpha}\, ]_S=0$. 
Now, if $x \in \alpha$ then 
$[\, (x,\gamma) \, , \, I_\alpha \, ]_S = (x,\gamma)-(x,\gamma)=0.$ 
Therefore, for all $\alpha \in \overline{Q}_1$, 
$[\, (x,\gamma)\, , \, I_\alpha \, ]_S =0$. 
If $l(\gamma) \geq 2$ then we will show that 
$[\, (x, \gamma)\, , \, I_\alpha \, ]_S$ belongs to $R$. 
Using the combinatorial bracket definition, 
$[\, (x, \gamma) \, , \, I_\alpha \, ]_S$ is a multiple of $(x,\gamma)$. 
Since $l(\gamma) \geq 2$, $[\, (x, \gamma) \, , \, I_\alpha \, ]_S$ is in $R$ 
for all $\alpha$ in $\overline{Q}_1$.  
We conclude that $[ \,k(Q_1 \parallel B) \, ,\,  I \,]_S \subseteq R$. 
From this inclusion  we infer that $I$ is an ideal. 
Moreover, $[ \, I \, , \,  I \, ]_S \subseteq R$, 
then I is solvable since $R$ is solvable.  
\end{proof}

\begin{lemma}\label{radqb}
Let $Q$ be a quiver without loops. 
Consider the Lie algebra ${k(Q_1 \parallel B)}$. 
Then 
\[
rad \, k(Q_1 \parallel B)= 
rad \, k(Q_1 \parallel Q_1) \oplus R 
\] 
\end{lemma}
\begin{proof}
The above lemma gives one inclusion. 
From Lemma  \ref{decoqb}, we know that 
$k(Q_1\parallel B) = k(Q_1 \parallel Q_1) \oplus R$. 
Let $w$ be in $k(Q_1 \parallel B)$, there exists 
$v$ in $k(Q_1 \parallel Q_1)$ and $z$ in $R$ 
such that $w=v+z$. 
In order to prove the other inclusion, let us show the following: 
if $w$ is in $rad \, k(Q_1 \parallel B)$ then $v$ is in 
$rad \, k(Q_1 \parallel Q_1)$. 
Consider the projection map: 
$$p:k(Q_1 \parallel B) \rightarrow k(Q_1 \parallel Q_1).$$
It is clear that $p$ is a Lie algebra epimorphism, therefore 
$$
p(rad\, k(Q_1 \parallel B)) \subseteq rad \, k(Q_1 \parallel Q_1). 
$$ 
Therefore, if $w$ is in $rad \, k(Q_1 \parallel B)$ then 
$v=p(w)$ is in $rad \, k(Q_1 \parallel Q_1)$. 
\end{proof}


\subsection{Radical square zero. }
Now, we deal with a 
particular case of monomial algebras: 
those of radical square zero. 
In this case, $Z$ is the set of all paths of length two, i.e. $Z=Q_2$. 
The set of vertices and arrows forms a basis of $kQ/<Q_2>$, i.e. 
$B=Q_0 \cup Q_1$.  
In \cite{cibils}, Cibils  
described the Hochschild cohomology groups for these algebras 
using a complex which coincides with the complex 
induced by the Happel-Bardzell resolution. 

In the next paragraph, we recall 
the computation of the first Hochschild cohomology group for monomial algebras of radical square zero. Let us remark that the vector space  
${k(Q_0 \parallel Q_0 \cup Q_1)}$ is isomorphic as a vector space to 
${k(Q_0 \parallel Q_0) \oplus k(Q_0 \parallel Q_1)}$. 
Therefore, the beginning of the Happel-Bardzell complex becomes 
\[
\scalebox{0.88}{$
0 \rightarrow 
k(Q_0 \parallel Q_0) \oplus k(Q_0 \parallel Q_1)
\stackrel{\psi_0}{\longrightarrow} 
k(Q_1 \parallel Q_0) \oplus k(Q_1 \parallel Q_1)
\stackrel{\psi_1}{\longrightarrow} 
k(Q_2 \parallel Q_0) \oplus k(Q_2 \parallel Q_1) . 
$}
\]
The differential can be restated as follows: 
\[
\begin{array}{ccccccccc}
\psi_0&=&\matrizDcero 
\, & \text{ and } &\, 
\psi_1&=&\matrizDuno
\end{array}
\]
where
\begin{equation}\label{Dzero}
\begin{array}{cccl}
D_0: & k(Q_0 \parallel Q_0) & \rightarrow & k(Q_1 \parallel Q_1) \\
\, & (e,e) & \mapsto & \displaystyle{
\sum_{a \in Q_1e} (a, a) \: \: - \sum_{a \in eQ_1} (a, a )} 
\end{array}
\end{equation}
and 
\[
\begin{array}{cccl}
D_1: & k(Q_1 \parallel Q_0) & \rightarrow & k(Q_2 \parallel Q_1) \\
\, & (a,e) & \mapsto & \displaystyle{
\sum_{b \in Q_1e} (ba,b ) \: \: + \sum_{b \in eQ_1} (ab, b )} . \\
\, &\, &\, &\,
\end{array}
\]

\begin{remark}
We already know that 
$HH^0(A)$ is the center of $A$. 
From the above complex, 
we deduce that the center of $A$ 
is $ker \, \psi_0$, which is equal to \linebreak 
$ker \, D_0 \oplus k(Q_0 \parallel Q_1)$. 
Let us remark that 
$D_0(\sum_{e \in Q_0} (e,e) )=0$. 
Then the element 
$\sum_{e \in Q_0} (e,e)$, 
which is the unit of $A$, is in $ker \, D_0$, 
which is not surprising since 
the unit of $A$ is always in its center. 
\end{remark}

\begin{lemma}\label{dimImD}
Let $Q$ be a quiver. The dimension of $Im \, D_0$ is $|Q_0|-1$. 
\end{lemma}
\begin{proof}
Suppose $Q$ is a loop or a multiple loops quiver  
then $D_0=0$ and $Im \, D_0$ has dimension $0$. 
So the result holds in this case since ${|Q_0|-1=0}$. 
Now, suppose $Q$ is not a loop nor a multiple loops quiver, then 
we will show that $ker \,  D_0$ is one dimensional. 
Let $x$ be in $ker \, D_0$, we write 
$x=\sum_{e \in Q_0} \, \lambda_e (e,e)$ with $\lambda_e$ in $k$. 
Since $D_0(x)=0$, 
$$
\sum_{e \in Q_0} \lambda_e \bigl( \sum_{a \in Q_1e} (a,a) - \sum_{a \in eQ_1} (a,a) \bigr )
$$
is zero. Let $a$ be an arrow such that $s(a) \ne t(a)$. 
Notice that the element $(a,a)$ appears in the above linear combination 
with coefficient $\lambda_{s(a)}-\lambda_{t(a)}$. 
Therefore $\lambda_{s(a)}=\lambda_{t(a)}$ for all $a$ such that $s(a) \ne t(a)$. 
We conclude that $\lambda_e=\lambda_{e'}$ for all $e,e'$ in $Q_0$ since $Q$ is connected. 
We infer that $ker \, D_0$ is the vector space generated by 
$\sum_{e \in Q_0} e$, the unit of $A$. 
Finally, the vector space $Im \, D_0$ 
has dimension $|Q_0|-1$ since 
$ker \, D_0$ is one dimensional. 
\end{proof}

In order to compute 
$HH^1(A)$, we have to determine the kernel of $\psi_1$ 
and the image of $\psi_0$. 
We will perform this computation in three  cases: 
first for the loop, 
then for the oriented cycle of length greater or equal two   
and finally for quivers that are not an oriented cycle.

For the {\bf{loop}}, let $e$ be the vertex and $a$ be the arrow. 
It is clear that $D_0=0$ and $D_1(a,e)=2(a^2,a)$. 
If $char k= 2$ then $D_1=0$ and 
therefore we conclude $$HH^1(A,A)=k(a,e) \oplus k(a,a).$$ 
If $chark \ne 2$ the map $D_1$ is clearly injective so 
$Ker \, D_1=0$.  So 
$Ker \, \psi_1=k(a,a)$. Since $Im \, \psi_0$ is zero,  
$$HH^1(A)=k(a,a).$$  

For the {\bf{oriented cycle}}, let $e_1, \dots e_N$ 
be the vertices and $a_1, \dots a_N$ be the arrows. 
We will suppose  $N \geq 2$ and $s(a_i)=e_i$. 
Then $D_1=0$, so $HH^1(A)$ is isomorphic to the quotient of  
$k(Q_1 \parallel Q_1)$ by $Im \, D_0$. Therefore, 
$$
HH^1(A) = \frac{\oplus_{i=1}^N k(a_i,a_i)}
{<(a_i,a_i)-(a_{i-1},a_{i-1})>_{i=2,\dots,N}} 
$$

For a quiver that is {\bf{not an oriented cycle}}, 
the map $D_1$ is injective, this was proven by Cibils 
in \cite{cibils}. Therefore, 
$$HH^1(A)= \frac{k(Q_1 \parallel Q_1)}{Im \, D_0}. $$ 

Since we have obtained the explicit computation of the first Hochschild cohomology group of a monomial algebra of radical square zero, 
we are able to study its Lie algebra structure. 
First, we fix some notation and we give some technical results. 

\begin{notation}
Denote $\chi(Q)$ the Euler characteristic 
of the underlying graph of the quiver $Q$, i.e. 
$$\chi(Q)=|Q_1|-|Q_0|+1.$$
Let us remark that 
if the underlying graph of $Q$ is a tree then $\chi(Q)=0$.
\end{notation}

\begin{lemma}\label{imd0rad}
Let $Q$ be a quiver that is not an oriented cycle.
Consider the Lie algebra $k(Q_1 \parallel Q_1)$. 
Then $Im \, D_0$ is an abelian ideal 
of $k(Q_1 \parallel Q_1)$. Therefore, 
$$ Im \, D_0 \subseteq rad \, k(Q_1 \parallel Q_1).$$

Moreover, if the underlying graph of 
$\overline{Q}$ is a tree then 
$$Im \, D_0 = rad \, k(Q_1 \parallel Q_1).$$
\end{lemma}
\begin{proof}
First notice that  
$$D_0(e,e)= \sum_{\alpha \in \overline{Q}_1e}  I_{\alpha} -
\sum_{\alpha  \in e\overline{Q}_1} I_{\alpha}.$$ 
In order to show that 
$Im \, D_0$ is an ideal of $k(Q_1 \parallel Q_1)$, 
let us compute first $[\, I_{\alpha}\, , \, (x,x') \, ]_S$. 
If $x \notin \alpha$ then it is clear that 
$[\, I_{\alpha} \, , \,  (x,x') \, ]_S=0$. 
If $x \in \alpha$ then 
$[\, I_{\alpha} \, , \,  (x,x')\, ]_S= 
(x,x')-(x,x')=0.$ 
We conclude that $[\, D_0(e,e) \, , \,  (x,x')\, ]_S=0$ for all $e \in Q_0$ and for all 
$(x,x') \in k(Q_1 \parallel Q_1)$. Therefore, 
$[\, D_0(e,e) \, , \, w \, ]_S=0$ for all $w \in k(Q_1 \parallel Q_1)$. 
From this computation, it is clear that $Im \, D_0$ is an abelian ideal. 
For the last statement, recall that 
$rad \, k(Q_1 \parallel Q_1)$ is equal 
to $\prod_{\alpha \in \overline{Q}} I_{\alpha}$, 
so its dimension is $|\overline{Q}_1|$. 
From the above remark $Im \, D_0$ has dimension $|Q_0|-1$.  
Since $\overline{Q}$ is a tree, $|\overline{Q}_1|=|Q_0|-1$ and therefore 
both ideals $Im D_0$ and $rad \, k(Q_1 \parallel Q_1)$,  
are equal. 
\end{proof}

 We will use the following lemma as a tool to compute the radical of the Lie algebras that we are dealing with. 

\begin{lemma}\label{lemmaradical}
Let $\g$ be a Lie algebra and $I$ be a solvable ideal. 
Then 
\[
rad \, \frac{\g}{I} = 
\frac{rad \, \g}{I} \quad . 
\] 
\end{lemma}

\begin{proof}
Let us consider the canonical projection 
${p:\g \twoheadrightarrow \g/I }$, 
therefore \linebreak 
$
{p(rad \, \g )
\subseteq 
rad \,(\g /I) }
$ 
since the image of a solvable ideal is solvable. 
Since $I$ is solvable, it is included in $rad \, \g$ and therefore 
\[
\frac{rad \, \g }{I}=p(rad \, \g) \subseteq 
rad \, \frac{\g}{I}.
\] 
In order to prove the lemma we have to prove 
the equality.
To do so,  we use the bijective correspondence between 
the ideals of the quotient $\g/I$ and the ideals of $\g$ that contain $I$. 
Let us suppose that $J$ is the ideal 
of $\g$ which contains $I$ such that
\[
rad \, \frac{\g}{I} = \frac{J}{I}.
\]
Clearly, $J$ contains 
$rad \, \g$ using the above inclusion.
We will prove that ${J=rad \, \g}$. 
It is enough to see that $J$ is solvable, 
since if this is true then $J$ is included in  
$rad \, \g$ and we obtain the result. 
We know that $J/ I$ is a solvable ideal of 
$\g / I$, so 
${\mathscr{D}}^l(J / I)=0$ for some $l$.
Let us also notice that 
\[
{\mathscr{D}}^l(J/I)=\frac{{\mathscr{D}}^l(J)+ I}{I}.
\]
Then ${\mathscr{D}}^l(J) \subseteq I$. 
This implies  
\[
[{\mathscr{D}}^l(J) \, , \,  {\mathscr{D}}^l(J)] 
\subseteq 
[I \, ,\, I]
\] 
and therefore we infer that  ${\mathscr{D}}^{l+l'}(J)=0$ for some $l'$. 
\end{proof}

\begin{remark}
Let $Q$ be a quiver that is not an oriented cycle.
We apply the above lemma 
to $\g=k(Q_1 \parallel Q_1)$ and to $I=Im \, D_0$. 
Then 
\[
rad \, \frac{k(Q_1 \parallel Q_1)}{Im \, D_0} = 
\frac{rad \, k(Q_1 \parallel Q_1)}{Im D_0} \quad . 
\] 
\end{remark}

\begin{notation}
We will denote $S$ the set of 
non-trivial equivalence classes:
\[
S=\{ \alpha \in \overline{Q}_1 \, \mid \, |\alpha|>1\} 
\]
\end{notation}

\begin{remark}
If $\alpha$ is in $S$ denote $sl_{|\alpha|}(k)$ the simple 
Lie algebra of $|\alpha| \times |\alpha|$ matrices of trace zero. 
It is clear that $\gl_{\alpha} / k\, I_{\alpha} $ is isomorphic as a
Lie algebra to $sl_{|\alpha|} (k)$ if the characteristic of the field is zero. 
\end{remark}

\begin{proposition}\label{resultadotecnico}
Assume that the field $k$ is algebraically closed and of characteristic zero. 
Let $Q$ be a quiver which is not an oriented cycle. Then,
\[
\frac{k(Q_1 \parallel Q_1)}{Im \, D_0} \cong 
\prod_{\alpha \in S} sl_{|\alpha|} (k) 
\, \times \, k \,^{\chi(\overline{Q})} 
\] 
as Lie algebras. 
\end{proposition}
\begin{proof}
First, we will compute the semisimple part. 
To do so we have to compute the quotient of $k(Q_1 \parallel Q_1)/ Im \, D_0$ 
by its radical. The radical is given by the quotient 
$rad \, k(Q_1 \parallel Q_1) / Im \, D_0$ 
using the above lemma. 
Therefore, it is clear that the semisimple part is 
$k(Q_1 \parallel Q_1) / rad \, k(Q_1 \parallel Q_1)$ 
which is isomorphic to $\prod_{\alpha \in S} sl_{|\alpha|} (k)$.  
The quotient $rad \, k(Q_1 \parallel Q_1) / Im \, D_0$ 
is isomorphic to the quotient of 
$\prod_{\alpha \in \overline{Q}_1} k \, I_{\alpha}$ by $Im \, D_0$ 
which is  in turn isomorphic to $k \,^{\chi(\overline{Q})}$. 
The last assertion follows from the fact that  
$\prod_{\alpha \in \overline{Q}_1} k \, I_{\alpha}$ 
is abelian of dimension $|\overline{Q}_1|$ and 
$Im \, D_0$ is of dimension $|Q_0|-1$. 
Once we have the semisimple part and the radical,  
we apply Levi's theorem which gives us the decomposition. 
We will show that the product is direct. 
Let $\overline{x}$ be in $k(Q_1 \parallel Q_1)/ Im \, D_0$ 
and $\overline{y}$ be in its radical where 
$x$ is in $k(Q_1 \parallel Q_1)$ 
and $y$ is in $\prod_{\alpha \in \overline{Q}_1} k \, I_{\alpha}$.  
Using the combinatorial bracket, $[x,y]_S=0$ and 
therefore $[\,\overline{x} \, ,\,\overline{y} \, ]_S=0$.
\end{proof}

The computation of $HH^1(A)$ and the study of the Lie algebra structure of $k(Q_1 \parallel Q_1)$ by $Im \, D_0$ provides the following result:

\begin{theorem}\label{hhuno}
Let $A=kQ/<Q_2>$ be a monomial algebra of radical square zero 
where $k$ is an algebraically closed field of characteristic zero and $Q$ is a finite connected quiver. 
Then 
\[
HH^1(A) \cong \prod_{\alpha \in S } sl_{|\alpha|} (k)
\; \times  \; k \,^{\chi(\overline{Q})} \quad .
\]
Therefore, $HH^1(A)$ is reductive. 
\end{theorem}

\begin{proof}
First, if $Q$ is the loop we have shown that 
$HH^1(A)=k(a,a)$, 
which is clearly isomorphic to $k$.   
Moreover, since $\chi(\overline{Q})=1$ and $S=\o$ 
we obtain the above isomorphism. 
Second, 
if $Q$ is an oriented cycle of length $\geq 2$, 
we have proved that $HH^1(A)$ is the quotient of 
$\oplus_{i=1}^N k(a_i,a_i)$ by the two sided ideal generated by elements of the form $(a_i,a_i)-(a_{i-1},a_{i-1})$, which is the image of $D_0$. 
The numerator is abelian of dimension $N$ (where $N$ is the length of the oriented cycle) 
and the denominator is of dimension $N-1$ (see Lemma \ref{dimImD}).
Then $HH^1(A)$ is one dimensional, 
therefore isomorphic to $k$.  Moreover, since 
$\chi(\overline{Q})=1$ and $S=\o$ 
we obtain the above isomorphism for the oriented cycle.  
Finally, if $Q$ is not the oriented cycle, then $HH^1(A)$ is the quotient of $k(Q_1 \parallel Q_1)$ by $Im \, D_0$. We apply the corollary \ref{resultadotecnico} to obtain the above isomorphism.  
\end{proof}

\begin{example}
The above result implies 
that $HH^1(k[x]/<x^2>)=k$, 
which is the one dimensional abelian Lie algebra. 
In \cite{happel} there is a description of the derivations of this algebra.
\end{example}
\begin{example}
Let $Q$ be the multiple arrows quiver: 
this is the quiver with two vertex and several arrows in the same direction between this two vertices. 
Assume that the number of arrows is greater or equal to two. 
If $A=kQ$, 
the above theorem implies that $HH^1(A)$ 
is isomorphic to $sl_r (k)$ 
where $r$ is the number of arrows. This computation can be obtained from the results in \cite{strametz}.
\end{example}
\begin{example}
Let $Q$ be the multiple loops quiver: 
this is the quiver which has one vertex and several loops. 
Assume that the number of loops is greater or equal to two, 
 $\overline{Q}$ is the one loop quiver.  
If $A=kQ/<Q_2>$, 
the above result implies that $HH^1(A)$ 
is isomorphic to $sl_r (k) \times k\cong gl_{r}(k)$ 
where $r$ is the number of
loops if $k$ is a field of characteristic zero and algebraically
closed. Denote $HH^1(A)_{ss}$ 
the semisimple part of a Lie algebra, 
this is the quotient by its solvable radical. 
Then $HH^1(A)_{ss}=sl_r (k)$. In a previous paper, \cite{sele}, we have obtained this result for $n=2$.
\end{example}

The next corollary of the above theorem gives  Strametz' conditions for semisimplicity. 
We will study those conditions in the next section. 
\begin{corollary}
Let $A=kQ/<Q_2>$ be a monomial algebra of radical square zero. 
Then $HH^1(A)$ is semisimple if and only if $S \ne \o$ and 
the underlying graph of the quiver $\overline{Q}$ is a tree. 
\end{corollary}


\subsection{Triangular complete monomial.}
In this section, 
we will study the Lie algebra structure of 
the first Hochschild cohomology of 
a triangular complete monomial algebra. 
Let $A=kQ/I$ be any finite dimensional algebra, 
where $I$ is an admissible ideal of the path algebra $kQ$. 

\begin{definition}[Triangular algebra]
If $Q$ has no oriented cycles, 
we say $A$ is a {\it{triangular algebra}}.
\end{definition}

\begin{definition}[Complete monomial algebra]
Let $A=kQ/<Z>$ be a monomial algebra. 
We say that $A$ is {\it{complete}} if and only if  
any path of length at least two which is parallel to a path in $<Z>$ 
is also in $<Z>$. 
\end{definition}

We compute the first Hochschild cohomology group 
of triangular complete monomial algebras 
using the complex induced from the Bardzell-Happel resolution. 
Since $A$ is triangular, 
the set $Q_0 \parallel B$ is in
fact $Q_0 \parallel Q_0$. 
Since $A$ is complete monomial 
then for all $p$ in $Z$ and 
$(a, \gamma)$ in $Q_1 \parallel B$,  $p^{(a,\gamma)}$ is in $<Z>$ 
by definition. 
Consider the map
$\psi_1:k(Q_1 \parallel B) \rightarrow k(Z \parallel B)$ 
appearing in the complex induced by the Bardzell-Happel resolution, 
we obtain that $\psi_1(a,\gamma)=0$ 
for all $(a, \gamma)$ in $Q_1 \parallel B$. 
Moreover, $Z \parallel B=Z \parallel Q_1$. 
Therefore, 
the complex induced
by the Bardzell-Happel projective resolution is:
\[
0 \longrightarrow k(Q_0 \parallel Q_0) 
\stackrel{\psi_0}{\longrightarrow} k(Q_1 \parallel B)
\stackrel{0}{\longrightarrow} k(Z \parallel Q_1)
\]
where the map $\psi_0$ 
is in fact the map $D_0$. 
Therefore, 
$$HH^1(A)= \frac{k(Q_1 \parallel B)}{Im \, D_0} \, .$$

The following lemma 
is analogue to Lemma  \ref{imd0rad} for 
radical square zero algebras. 
We show that the denominator 
of the above quotient is 
in the radical of the numerator. 

\begin{lemma} 
Let $Q$ be a quiver without oriented cycles. 
Consider the Lie algebra $k(Q_1 \parallel B)$. 
Then $Im \, D_0$ is an abelian ideal, 
therefore 
$$Im \, D_0 \subseteq rad \, k(Q_1 \parallel B).$$
\end{lemma}

\begin{proof}
Let $e$ be in $Q_0$ and let 
$(x, \gamma)$ be in $Q_1 \parallel B$, 
i.e  the arrow $x$ is parallel to 
$\gamma$, a path in $B$ of length $n$. 
We will show that $[\, D_0(e,e) \, , \,  (x, \gamma)\, ]_S=0$. 
Since 
$$D_0(e,e)=\sum_{\alpha \in \overline{Q}_1e}  I_{\alpha} -
\sum_{\alpha  \in e\overline{Q}_1} I_{\alpha}.$$  
we will compute $[\, I_\alpha \, ,\,  (x,\gamma)\, ]_S$ 
for $\alpha$ in $\overline{Q}_1$ such that 
$s(\alpha)=e$ or $t(\alpha)=e$.
Assume $\gamma=y_1 \cdots y_i \cdots y_n$. 
Let us denote $e_0=t(y_1)=t(x)$, 
$e_i=s(y_i)=t(y_{i+1})$ for $i=1, \dots,n-1$ 
and $e_n=s(y_n)=s(x)$. 
It is clear that all $e_i$ for $i=0,\dots,n$ are different. 
First, suppose $e$ is a vertex with $e \ne e_i$ 
for all $i=0,\cdots, n$. 
Let $a$ be an arrow such that $s(a)=e$
(or $t(a)=e$), 
$[\, (a,a)\, ,\,  (x,\gamma)\, ]_S=0$ since 
$a$ is neither $x$ nor any $y_i$.  
Therefore $[\, I_\alpha \, , \, (x,\gamma)\, ]_S=0$ for all 
$\alpha$ in $\overline{Q}_1e$ and for all $\alpha$ in $e\overline{Q}_1$. 
We conclude that 
if $e \ne e_i$ then $[\, D_0(e,e) \, , \,  (x, \gamma) \, ]_S=0$. 
Suppose now that $e=e_i$ for some $i=1,\dots, n-1.$ 
Let us remark that 
$y_i$ which is in the decomposition of $\gamma$, 
is an arrow whose source is $e$. 
Notice also that the arrow $y_{i+1}$ which 
is in the decomposition of $\gamma$ too, is 
an arrow whose target is $e$. 
A simple calculation gives us that: 
\[
\begin{array}{lcc}
\, [\, ( y_i, y_i) \, , \, (x,\gamma) \, ]_S& = & (x,\gamma) \\
\, [\, ( y_{i+1} \, , \,  y_{i+1}),(x,\gamma) \, ]_S& = & (x,\gamma) .
\end{array}
\]
Denote by $\alpha_i$ the equivalence class of $y_i$ and 
denote also by $\alpha_{i+1}$  the equivalence class of $y_{i+1}$.
Observe that $\alpha_i$ is in $\overline{Q}_1e$ while 
$\alpha_{i+1}$ is in $e\overline{Q}_1$.  
It is clear that $[\, I_{\alpha_i} \, , \, (x,\gamma)\, ]_S=
[\, I_{\alpha_{i+1}} \, , \, (x,\gamma)\,]_S = (x, \gamma)$. 
Now, let $\alpha$ be in $\overline{Q}_1$ such that  
$s(\alpha)=e$ (resp. $t(\alpha)=e)$, but $\alpha$ is different 
from $\alpha_i$ (resp. $\alpha_{i+1})$.  
Then  $[\, I_{\alpha}\, , \, (x,\gamma)\, ]_S=0$ 
since for all $a$ in $\alpha$, $a$ is neither $x$ nor any $y_i$. 
We can conclude now that if $e=e_i$ for some $i=1,\dots,n-1$, then 
$$[\, D_0(e,e)\, , \, (x,\gamma)\, ]_S=
[\, I_{\alpha_i} \, , \, (x,\gamma) \, ]_S - [\, I_{\alpha_{i+1}} \, , \, (x,\gamma)\, ]_S=
(x,\gamma)-(x,\gamma)=0.$$ 
Finally, suppose $e=e_0$;
both arrows $y_n$ and $x$ have source $e$. 
A simple calculation gives us that: 
\[
\begin{array}{lcc}
\, [\, ( y_n, y_n) \, , \, (x,\gamma) \, ]_S& = & (x,\gamma) \\
\, [\, ( x, x) \, , \, (x,\gamma)\, ]_S& = & -(x,\gamma) .
\end{array}
\]
Denote $\alpha_n$ the equivalence class of $y_n$ 
and denote  $\alpha_x$ the equivalence class of $x$. 
Both $\alpha_n$ and $\alpha_x$ are in $\overline{Q}_1e$. 
It is clear that $[\, I_{\alpha_n} \, , \, (x,\gamma)\, ]_S=(x,\gamma)$ and that 
$[\, I_{\alpha_x}\, , \, (x,\gamma)\, ]_S=-(x,\gamma).$ As for the previous case, 
for all $\alpha $ in $\overline{Q}_1$ such that  
$s(\alpha)=e$ (resp. $t(\alpha)=e)$, 
but different from  $\alpha_n$ and from $\alpha_x$, 
we infer  $[\, I_{\alpha}\, , \, (x,\gamma)\,]_S=0$. 
We can conclude now that if 
$e=e_0$, 
$$[\, D_0(e,e) \, , \, (x,\gamma) \, ]_S=
[\, I_{\alpha_n} \, , \, (x,\gamma)\, ]_S + [\, I_{\alpha_{x}} \, , \, (x,\gamma)\, ]_S=
(x,\gamma)-(x,\gamma)=0.$$ 
A similar argument, give us that for $e=e_n$ 
$$[\, D_0(e,e) \, , \, (x,\gamma)\, ]_S=-[\, I_{\alpha_0} \, , \, (x,\gamma)\,]_S 
- [\, I_{\alpha_{x}} \, , \, (x,\gamma)\, ]_S=
-(x,\gamma)+(x,\gamma)=0$$
where $\alpha_0$ is the equivalence class of $y_1$. 
Both  $\alpha_0$ and $\alpha_x$ are in $e\overline{Q}_1.$
\end{proof}

\begin{lemma}
Let $Q$ be a quiver without oriented cycles. 
Consider the Lie algebra $k(Q_1 \parallel B)$. 
Then  
\[
rad \, \frac{k(Q_1 \parallel B)}{Im \, D_0} = 
\frac{rad \, k(Q_1 \parallel B)}{Im D_0} \quad . 
\] 
\end{lemma}

\begin{proof} 
We apply Lemma  \ref{lemmaradical}.  
\end{proof}

Recall that $R$ is the solvable ideal of $k(Q_1 \parallel B)$ given by the direct sum of all $k(Q_1 \parallel Q_i \cap B)$ where $i$ goes  from $2$ to $N$, the maximum of 
all length of non-zero paths in A. 

\begin{proposition}\label{triangular}
Let $A=kQ/<Z>$ be a triangular complete monomial algebra where 
$k$ is an algebraically closed field of characteristic zero. Then
\[
HH^1(A) \cong 
\prod_{\alpha \in S} sl_{|\alpha|} (k) 
\, \rtimes \, \big( k \,^{\chi(\overline{Q})}
\oplus R \big). 
\]
Moreover 
$[\, sl_{|\alpha|} (k) \, , \, k \,^{\chi(\overline{Q})}
\oplus R  \, ] \subseteq R.$ 
\end{proposition}
\begin{proof}
As we did for the radical square zero case, 
first we compute the semisimple part. So we have to compute the quotient of 
$k(Q_1 \parallel B) / Im \, D_0$ by its radical, which is 
$rad \, k(Q_1 \parallel B) / Im \, D_0$ using the above lemma. Clearly, the semisimple part is isomorphic to the quotient 
of $k(Q_1 \parallel B)$ by $rad \, k(Q_1 \parallel B)$. 
Recall that $k(Q_1 \parallel B)$ is equal to 
$k(Q_1 \parallel Q_1) \oplus R$ and 
$rad \, k(Q_1 \parallel B) $ is equal to 
${rad \, k(Q_1 \parallel Q_1) \oplus R}$. 
Therefore, 
the semisimple part of $HH^1(A)$ is 
$\prod_{\alpha \in S} sl_{|\alpha|} (k)$. 
To compute the radical of $HH^1(A)$, we have to compute the quotient of $rad \, k(Q_1 \parallel Q_1) \oplus R$ by $Im \, D_0$. Since $Im \, D_0$ is an abelian ideal of $k(Q_1 \parallel Q_1)$, 
then the radical is precisely 
$k \,^{\chi(\overline{Q})} \oplus R$. Using the Levi decomposition theorem we obtain the result.
\end{proof}

\begin{corollary}
Let $A$ be a triangular complete monomial algebra. Then  
$HH^1(A)$ is semisimple if and only if $\overline{Q}$ is a tree and $S$ is not an empty set.  
\end{corollary}

\section{Semisimplicity and vanishing Hochschild cohomology.}

In \cite{strametz}, Strametz gave necessary and
sufficient conditions for the semisimplicity of the Lie algebra $HH^1(A)$ 
when $A$ is a monomial algebra. 
If we assume that the field is algebraically closed and of characteristic zero,  
her theorem states that $HH^1(A)$ is semisimple 
if and only if the following conditions are satisfied: 
\begin{itemize}
\item[-] the underlying graph of the quiver $\overline{Q}$ is a tree, 
\item[-] there exists a non trivial class in $\overline{Q}_1$ and 
\item[-] the ideal $<Z>$ is completely saturated. 
\end{itemize}
The aim of this section is to prove the following result: 
let $A$ be a monomial algebra with 
$HH^1(A)$ semisimple then 
$HH^n(A)=0$ for all $n \geq 2$.
To do so we will use the above stated 
conditions for semisimplicity. 
We begin recalling  
the definition of completely saturated ideal, 
and we prove that under the assumption 
that the underlying graph of the quiver $\overline{Q}$ is a tree,
the completely saturated condition is equivalent to $Z$ being closed by parallel paths. 
Therefore, we are able to restate Strametz's conditions. 
Then we proceed as follows:  
we assume the above conditions in order to compute 
the complex from the Bardzell-Happel resolution. 
Then we  prove that the Hochschild cohomology groups are zero in degrees greater than or equal to two. 

\subsection{Completely saturated condition.}

\begin{definition}[Completely Saturated]
Let $a \parallel b$ be two parallel arrows. 
We say that $a$ and $b$ are equivalent if 
$p^{(a,b)}=0=p^{(b,a)}$ for all $p$ in $Z$. 
The ideal $<Z>$ is called {\it{completely saturated}} 
if all parallel arrows are equivalent.   
\end{definition}

The next lemma gives a technical condition to determine if an ideal generated by paths is completely saturated. 

\begin{lemma}\label{saturado}
An ideal $<Z>$ is completely saturated if and only if 
for any path $p$ in $Z$ the following condition is verified: 
for each arrow $p_i$ in the expression of $p$ and 
for any arrow  $a$ parallel to $p_i$, 
the path  $p \diamantei a$ is in $<Z>$. 
\end{lemma}
\begin{proof} 
Let $p=p_1 \cdots p_i \cdots p_n$ be a path in $Z$ and $a,b$ two parallel arrows. 
Recall that the element $p^{(a,b)}$ in $A$ is given as follows:
\[
p^{(a,b)}=\sum_{i=1}^n 
\delta_{p_i}^a \chi_B (p \diamantei b) p\diamantei b
\] 
where $\delta_{p_i}^{a}$ is the Kronecker delta 
and $\chi_B$ is the characteristic function. 
Remark that $p^{(a,a)}=0$ for all $a$ in $Q_1$.   
Moreover, suppose  
$a$ is not parallel to any arrow $p_i$, then $p \diamantei a=0$ for all $i$ 
and therefore $p^{(a,-)}=0$. 
Now, suppose $a \ne b$ and that 
$a$ is parallel to some arrow in the path $p$. 
In this case, we will show that all paths 
which are summands of $p^{(a,b)}$ are different. 
Let $i,j$ be two natural numbers from $1$ to $n$ 
such that $p_i=a=p_j$.  
If $p \diamantei b=p_1 \cdots a \cdots p_n=p \diamantej a$ 
then $i=j$ since $a \ne b$. 
Therefore all paths $p \diamantei b$ 
in the above sum are different.

($\Rightarrow$) 
Let $a$ be an arrow such that $a \parallel p_i$ for some $i$. 
If $a=p_i$ then $p \diamantei a=p$, which is clearly in $<Z>$. 
Suppose  $a \ne p_i$. 
By hypothesis, all pairs of parallel arrows  
are equivalent, in particular $(p_i,a)$. So   
$$0=p^{(p_i,a)}=\sum_{j=1}^n 
\delta_{p_j}^{p_i} \chi_B (p \diamantej a) p\diamantej a. $$ 
We have shown that all paths on the right side of the formula are different. 
Therefore $\chi_{B}(p \diamantej a) \: p \diamantej a=0$ for all $j$. 
We conclude that $p \diamantei a$ is in $<Z>$. 

($\Leftarrow$) Let $a \parallel b$, 
we will show that they are equivalent, 
i.e. $p^{(a,b)}=0=p^{(b,a)}$.  
Using the above formula, 
it is clear that we are only interested in  
arrows $p_i$ which are equal to $a$ or $b$. 
Consider these $p_i$'s, evidently $p_i$ is parallel to $a$ and to $b$.   
By hypothesis, the paths $p \diamantei b$ and 
$p \diamantei a$ are in $<Z>$ for all $p_i$ that is equal to $a$ or to $b$. 
Then $\chi_B(p \diamantei a)=0=\chi_B(p \diamantei b)$. 
Using the above formula we  deduce 
that $p^{(a,b)}=0$ and $p^{(b,a)}=0$.
\end{proof}

\begin{lemma}\label{acompletezsaturated}
Let $Q$ be a quiver and let $Z$ be a minimal set of paths of length at least two. 
If $A=kQ/<Z>$ is a complete monomial algebra then $<Z>$ is completely saturated. 
\end{lemma}
\begin{proof}
Let $p$ be in $Z$, $p_i$ an arrow in the expression of $p$, 
and $a$ an arrow parallel to $p_i$. Since $A$ is a complete monomial algebra and 
$p \parallel p \diamantei a$,  $p \diamantei a_i$ is in $<Z>$. 
By the above lemma we conclude that $<Z>$ is completely saturated. 
\end{proof}

\subsection{Restatement of semisimplicity conditions.}
The objective in this section is to prove 
that the condition of "completely saturated" 
can be replaced  by  "closed under parallel paths" or by $A$ being "complete monomial" 
under the assumption that the underlying graph of $\overline{Q}$ is a tree. 

\begin{definition}[Closed under parallel paths]
A set of paths $Z$ is called {\it{closed under parallel paths}} if 
and only if for any path $p$ in $Z$ the following condition is verified: 
if $q$ is a path parallel to $p$ then $q$ is in $Z$. 
\end{definition}

\begin{example}
Let $Q$ be the following quiver:
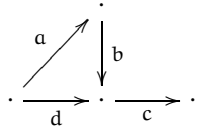
\[
\xymatrix{
\, & \cdot \ar[d]^b & \, \\
\cdot \ar[ur]^a \ar[r]_d & \cdot \ar[r]_c& \cdot  
}
\]
\begin{itemize}
\item[$1)$] If $Z=\{ cb \}$ then $Z$ is clearly closed under parallel paths. The algebra  
${A=kQ / <Z>}$ is not complete monomial  
since $cba$ is in $<Z>$ but $cd$ which is parallel to $cba$ is not in $<Z>$. 
\item[$2)$] If $Z=Q_2$ then $A=kQ/<Z>$ is a complete monomial algebra. The set $Z$ 
is not closed under parallel paths since $d \parallel ba$ but $d$ is not in $Z$. 
\item[$3)$] If $Z=\{ cd, cba \}$ then $Z$ is closed under parallel paths and 
${A=kQ/<Z>}$ is a complete monomial algebra. 
\item[$4)$] If $Z=\{ ba \}$ then $A=kQ/<Z>$ is not complete monomial since 
$cba$ is a path in $<Z>$ parallel to $cd$ which is not in $<Z>$.   
The set $Z$ is not closed under parallel paths.
\end{itemize}
\end{example}

\begin{lemma}\label{saturado2}
Let $Z$ be a minimal set of paths of length at least two. 
If the set $Z$ is closed under parallel paths, 
then the ideal $<Z>$ is completely saturated. 
\end{lemma}
\begin{proof}
Let   
$p$ be an element of $Z$ whose expression in arrows is 
$p_1 \cdots p_i \cdots p_n$.  
If an arrow $a$ is parallel to some $p_i$ then the path 
$p$ is clearly parallel to the path obtained 
by replacing $p_i$ with $a$, which is $p \diamantei a$. 
Since $Z$ is closed under parallel paths, 
then $p \diamantei a$ is in $<Z>$ for all $i$. 
\end{proof}

The converse of the above implication is not always true. 
For example, let $Z=Q_n$ where $n>1$, 
the ideal  $<Z>$ is completely saturated but 
$Z$ is not necessarily closed under parallel paths. 
For instance, let $Q$ be the quiver of the above example and 
let $Z=Q_2$. 
Notice that $<Z>$ is completely saturated and 
$Z$ is not closed under parallel paths. 
If the underlying graph of $\overline{Q}$ is a tree then the converse holds.  
Indeed, under this assumption, parallel paths are as follows. 

\begin{lemma} 
Assume that the underlying graph of 
$\overline{Q}$ is a tree. 
Let $\alpha=a_1 \cdots a_n$ 
and $\beta=b_1 \cdots b_m$ be 
parallel paths. Then $n=m$ and  $a_i \parallel b_i$.
\end{lemma}
\begin{proof} 
A pair of parallel paths in $Q$ provides a pair of parallel paths in 
$\overline{Q}$. Since $\overline{Q}$ is a tree, the former are equal, hence the original paths only differ by parallel arrows.
\end{proof}

The following proposition allows us to restate Strametz's theorem by replacing "completely saturated" by "closed under parallel arrows". 

\begin{proposition}\label{cerradoparalelos}
Let $Q$ be a quiver and 
$Z$ be a minimal set of paths of length at least two. 
Assume that the underlying graph of $\overline{Q}$ is a tree. 
If $<Z>$ is a completely saturated ideal then $Z$ is closed under parallel paths.
\end{proposition}

\begin{proof}
Let $p=p_1 \cdots p_n$ be a path in $Z$
and $q$ be a path parallel to $p$. 
Using the above lemma, 
$q=q_1 \cdots q_n$ and  
$q_i \parallel p_i$ for all $i$.
Since we suppose 
$<Z>$ is completely saturated, 
for each $i$ from $1$ to $n$ and 
$a$ parallel to $p_i$, the path 
$p \diamantei a$ is in the ideal 
$<Z>$ (Lemma \ref{saturado}). 
Let us notice that it is enough 
to prove that for any path $p$ in $Z$, 
the path $p \diamantei a$ is in $Z$.  
Once we have shown this, 
we can set $p(0)=p$ 
and for $i=1,\dots,n$ 
we can set $p(i)=p(i-1) \diamantei q_i$. 
Then each $p(i)$ will be in $Z$ and in particular $p(n)=q$. 
So let us prove that for any path $p$ in $Z$, 
and for all $a\parallel p_i$, the path   
$p \diamantei a$ is in $Z$:  
by induction on $l(p)$, 
the length of the path $p$. 
Let us suppose $l(p)=2$, 
then $p=p_1p_2$, and 
let $a \parallel p_i$. 
Since $<Z>$ is completely saturated, 
$p \diamantei a$ is in $<Z>$. 
This means that 
$p \diamantei a=\alpha p' \beta$ 
where $p'$ is in $Z$, $\alpha$ and $\beta$ are paths. 
Then it is easy to see that $\alpha$ and $\beta$ are trivial paths otherwise $p'$ is in $Q_1$, 
which is not possible since $Z \cap Q_1=\o$. 
Now, let us suppose $l(p)= n > 2$.  
Let $p=p_1\cdots p_n$ be in $Z$ and let $a \parallel p_i$. 
Since $<Z>$ is completely saturated, 
then $p \diamantei a$ is in $<Z>$, therefore 
$p \diamantei a=\alpha p' \beta$ where $p'$ is in $Z$. 
We have that $p'$ is parallel to 
$p_{j_1} \cdots p_{j_n}$ since the underlying graph of 
$\overline{Q}$ is a tree. 
If $\alpha$ or $\beta$ are non trivial paths, 
then $l(p') < n$. 
By the induction hypothesis  
$p_{j_1} \cdots p_{j_n}$ is in $Z$ since 
it is parallel to $p'$ which is in $Z$. 
But $Z$ is minimal so 
this is a contradiction. 
Therefore, $\alpha$ and $\beta$ are trivial paths. 
We thus obtain that $p \diamantei a$ is in $Z$. 
\end{proof}

\begin{corollary} 
Let $A=kQ/<Z>$ where 
$kQ$ is the path algebra of the quiver $Q$ and $Z$ is a minimal set of paths of length at least two. 
Then, $HH^1(A)$ is semisimple if and only if 
$Z$ is closed by parallel arrows, the underlying graph of $\overline{Q}$ is a tree and the set $S$ is not empty. 
\end{corollary}

We also have following result.

\begin{proposition}
Let $Q$ be a quiver and 
$Z$ be a minimal set of paths of length at least two. 
Assume that the underlying graph of $\overline{Q}$ is a tree. 
The following statements are equivalent:
\begin{itemize}
\item[$(1)$] $A=kQ/<Z>$ is a complete monomial algebra.
\item[$(2)$] $<Z>$ is a completely saturated ideal. 
\item[$(3)$] $\psi_1=0$
\item[$(4)$] $Z$ is closed under parallel paths.
\end{itemize}
\end{proposition}

\begin{proof}
(1 $\Rightarrow$ 2) see Lemma \ref{acompletezsaturated}.

(2 $\Rightarrow$ 3) from the formula given by (\ref{psiuno}) and under the hypothesis $\overline{Q}$ is a tree, 
the map $\psi_1$ is defined from $k(Q_1 \parallel Q_1)$ to $k(Z \parallel B)$. 
Since $Z$ is completely saturated, $\psi_1(a,b)=0$ for all $(a,b)$ in $Q_1 \parallel Q_1$. 

(3 $\Rightarrow$ 4) Let $p=p_1 \cdots p_n$ be in $Z$ 
and let $a$ be an arrow parallel to $p_i$.  
First, let us remark the following: 
since $\psi_1(p_i,a_i)=0$, 
for any $q$ in $Z$, 
$q^{(p_i,a_i)}=0$. 
In the particular case of $q=p$, 
$p^{(p_i,a_i)}=0$  which implies that  
$p \diamantei a_i$ is in $<Z>$. 
Therefore, $<Z>$ is completely saturated since  
the condition of Lemma \ref{saturado} is satisfied. 
Since the underlying graph of 
$\overline{Q}$ is a tree and $Z$ is completely saturated, 
$Z$ is closed under parallel paths using the 
Proposition \ref{cerradoparalelos}.  

(4 $\Rightarrow$ 1) 
Let $p$ be a path in $<Z>$ and $q$ be a path parallel to $p$. 
Assume $p=\alpha p' \beta$ where $p'$ is in $Z$ and 
$\alpha$ and $\beta$ are paths.  
Since the underlying graph of $\overline{Q}$ is a tree, 
$q=\alpha' q' \beta'$ where 
$\alpha \parallel \alpha'$, 
$\beta \parallel \beta'$ 
and  $p' \parallel q'$. 
By hypothesis, $q'$ is in $Z$, therefore $q$ is in $<Z>$. 
\end{proof}
 
\begin{corollary} Let $A$ be a monomial algebra. 
If $HH^1(A)$ is semisimple then $A$ is a complete monomial algebra.  
\end{corollary}
\subsection{Vanishing of the Hochschild cohomology.}
In the sequel, we will assume that the characteristic of the field is zero. 
Let $A$ be a finite dimensional monomial algebra. 
In this section, 
we prove that if $HH^1(A)$ is
semisimple   
then the Hochschild cohomology groups vanish in higher degrees. 
The principal tool is the 
Happel-Bardzell projective resolution. 
Let us recall some facts about this resolution. 

In \cite{happel}, Happel provides the projectives for a minimal projective resolution of a finite dimensional 
$k$-algebra over its enveloping algebra. 
Then in \cite{Bardzell}, Bardzell describes the projective modules for monomial algebras in terms 
of the combinatorics of $A$ and he describes  the morphisms of the resolution. 
\begin{notation}
The Happel-Bardzell minimal projective resolution for monomial
algebras given in \cite{Bardzell} is denoted by:  
\[
\bold{B \,=}\quad
\cdots \rightarrow P_{n+1}  
\rightarrow \, P_n 
\rightarrow \, \cdots
P_2 \, 
\rightarrow \, P_1  \, 
\rightarrow \, P_0  \,                          
\stackrel{\mu}{\rightarrow}  
A \rightarrow \, 0.
\]
The projective modules and morphisms are given explicitly in terms
of the quiver and the set of paths $Z$. 
The construction is rather technical, we provide a sketch of it: 
\[
\begin{array}{ccc}
P_0&=&A \tensorQ A \\
P_1&=&A \tensorQ kQ_1 \tensorQ A \\
P_2&=&A \tensorQ kZ \tensorQ A \\
P_n&=&A \tensorQ  kAP_{n} \tensorQ A
\end{array}
\] 
where $AP_n$ is a set of paths constructed by induction: 
for $n>2$ and for each path $p$ in $Z$, 
an {\it{associated sequence}}, 
$(p,r_2,\cdots,r_{n-1})$, 
of $n-1$ paths in $Z$ is given, 
then for each associated sequence 
there is a path  
with $s(p)$ as source and $t(r_n)$ 
as target associated to it. Then $AP_n$ is the 
collection of all those paths. 
The definition of "associated sequence" 
and the construction of paths from 
this associated sequences  
is explicitly given in \cite{Bardzell}. 
For the purpose of this paper 
we just need the following property.
\end{notation}

\begin{lema}[\cite{Bardzell}]
Let $p^n \in AP_{n}$. The set 
\[
Sub(p^n)=\{ \, 
p^{n-1} \in  AP_{n-1} \, \mid \, p^{n-1} \text{ divides } p^n 
\, \}
\]
contains two paths $p_o^{n-1}$ and $p_t^{n-1}$ where 
$s(p_o^{n-1})=s(p^n)$ and  
$t(p_t^{n-1})=t(p^n)$. 
Furthermore, if $n$ is odd then $Sub(p^n)=\{ p_o^{n-1},p_t^{n-1}\}$
\end{lema} 

The above lemma will enable us to compute the complex obtained 
from the Happel-Bardzell projective resolution for monomial algebras whose first Hochschild cohomology group is semisimple.

\begin{lemma}
Let $A=kQ/<Z>$ be a monomial algebra. 
Assume 
$Z$ is closed under parallel paths and that 
the underlying graph of $\overline{Q}$ is a tree. 
Let $p^n$ be a path in $AP_n$. 
Then any path parallel to $p^n$ is in the ideal $<Z>$. 
\end{lemma}

\begin{proof}
The proof is by induction. 
For $n=2$ the statement is true since 
$Z$ is closed under parallel paths. 
Now, let us suppose $n > 2$. 
Let $p^n$ be a path in $AP_n$. 
By Bardzell's lemma, 
$p^n=Lp_o^{n-1}$ 
where $p_o^{n-1}$ is in $AP_{n-1}$ and $L$ is some path. 
If $\alpha$ is a parallel path to $p^n$  
then $\alpha=L'p'$ where 
$L' \parallel L$ and $p'\parallel p_o^{n-1}$ 
since $\alpha$ is obtained by replacing parallel arrows in $p^n$ 
since $\overline{Q}$ is a tree. 
By the inductive hypothesis, $p'$ is in $<Z>$ 
since it is parallel to a path in $AP_{n-1}$. 
Therefore $\alpha$ is in $<Z>$.    
\end{proof}

\begin{theorem}
Let $A=kQ/<Z>$ be a finite dimensional monomial algebra where  
$k$ is a field of characteristic zero.
If the underlying graph of $\overline{Q}$ is a tree and $Z$ is closed under parallel paths 
then 
\begin{itemize}
\item[-] $HH^0(A)=k$  
\item[-] $HH^1(A)= \prod_{\alpha \in S } sl_{|\alpha|} (k)$ and 
\item[-] $HH^n(A)=0$ for all $n\geq 2$.
\end{itemize}
\end{theorem}

\begin{proof} 
Let us remark that $Z \parallel B$ is empty since $Z$ is closed
under parallel paths and the elements of $B$ form a basis of $A$. 
In general $AP_n \parallel B$ 
is empty for $n\geq2$ using the above lemma. 
The complex obtained after applying the functor $Hom_{A^e}(-,A)$ 
to the Happel-Bardzell resolution is isomorphic to the following complex:
\[
0 \rightarrow k(Q_0 \parallel Q_0) 
\stackrel{\psi_0}{\longrightarrow} k(Q_1 \parallel Q_1) 
\stackrel{\psi_1}{\longrightarrow} 0
\rightarrow 0 \rightarrow \cdots 
\, \rightarrow \, 0 \,\rightarrow \, \cdots 
\]
We deduce that $HH^0(A)=k$ and 
$HH^n(A)=0$ for $n \geq 2$. 
We use Proposition \ref{triangular} 
to describe the first Hochschild cohomology group. 
\end{proof}

Notice that  under the hypotheses of the above theorem, 
if $\overline{Q}=Q$ then $Q$ is a tree and $S$ is clearly empty. 
Therefore $HH^1(A)$ is zero, which is a result by Bardzell and Marcos \cite{bardzellmarcos}. 

Finally, we have the following result:
\begin{corollary}
Let $A=kQ/<Z>$ be a finite dimensional monomial algebra. 
If $HH^1(A)$ is semisimple then  
${HH^n(A)=0}$ for all $n \geq 2$. 
\end{corollary}

\begin{proof}
Since $HH^1(A)$ is semisimple then $A$ is complete monomial algebra and the underlying graph of $\overline{Q}$ is a tree. We apply the above theorem in order to obtain that $HH^n(A)=0$ for $n \geq2$.
\end{proof}

\bibliography{bibliographie}

\providecommand{\bysame}{\leavevmode\hbox to3em{\hrulefill}\thinspace}
\providecommand{\MR}{\relax\ifhmode\unskip\space\fi MR }
\providecommand{\MRhref}[2]{%
  \href{http://www.ams.org/mathscinet-getitem?mr=#1}{#2}
}
\providecommand{\href}[2]{#2}
\begin{thebibliography}{Hap89}

\bibitem[Bar97]{Bardzell}
M.J. Bardzell, \emph{The alternating syzygy behavior of monomial algebras}, J.
  Algebra \textbf{188} (1997), no.~1, 69--89.

\bibitem[BM98]{bardzellmarcos}
M.J. Bardzell and E.N. Marcos, \emph{Induced boundary maps for the cohomology
  of monomial and {A}uslander algebras}, Algebras and modules, {II}
  ({G}eiranger, 1996), CMS Conf. Proc., vol.~24, Amer. Math. Soc., Providence,
  RI, 1998, pp.~47--54.

\bibitem[Cib98]{cibils}
C.~Cibils, \emph{Hochschild cohomology algebra of radical square zero
  algebras}, Algebras and modules, II (Geiranger, 1996), CMS Conf. Proc.,
  vol.~24, Amer. Math. Soc., Providence, RI, 1998, pp.~93--101.

\bibitem[Ger63]{gersten}
M.~Gerstenhaber, \emph{The cohomology structure of an associative ring}, Ann.
  of Math. (2) \textbf{78} (1963), 267--288.

\bibitem[Hap89]{happel}
D.~Happel, \emph{Hochschild cohomology of finite-dimensional algebras},
  S\'eminaire d'Alg\`ebre Paul Dubreil et Marie-Paul Malliavin, 39\`eme Ann\'ee
  (Paris, 1987/1988), Lecture Notes in Math., vol. 1404, Springer, Berlin,
  1989, pp.~108--126.

\bibitem[SF08]{sele}
S.~Sanchez-Flores, \emph{The {L}ie module structure on the the {H}ochschild
  cohomology groups of monomial of algebras of radical square zero}, J. Algebra
  \textbf{320} (2008), 4249--4269.

\bibitem[Str06]{strametz}
C.~Strametz, \emph{The {L}ie algebra structure on the first {H}ochschild
  cohomology group of a monomial algebra}, J. Algebra Appl. \textbf{5} (2006),
  no.~3, 245--270.

\end{thebibliography}
\bibliographystyle{amsalpha}
\end{document}